\documentclass[letterpaper, 10 pt, conference]{ieeeconf}  

\IEEEoverridecommandlockouts                              

\usepackage{graphicx}
\usepackage{amsmath}
\usepackage{amssymb}
\usepackage{comment}
\usepackage{color}
\usepackage{float}
\usepackage{mathrsfs}
\usepackage{hyperref}
\usepackage{cite}
\usepackage{enumitem,mathtools}

\usepackage{amsfonts,latexsym,amsmath,amssymb,amsthm}
\usepackage[mathscr]{eucal}
\usepackage{graphicx,color}
\usepackage{comment}
\usepackage{hyperref}
\usepackage[margin = 1in]{geometry}

\newcommand{\overbar}[1]{\mkern 1.5mu\overline{\mkern-1.5mu#1\mkern-1.5mu}\mkern 1.5mu}

\makeatletter
\let\theoremstyle\@undefined                        
\makeatother

\usepackage{amsthm}

\newtheorem{nntheorem}{\bf Theorem}
\newtheorem{nnassumption}{\bf Assumption}
\newtheorem{nndefinition}{\bf Definition}
\newtheorem{nnlemma}{\bf Lemma}
\newtheorem{nncorollary}{\bf Corollary}
\newtheorem{nnproposition}{\bf Proposition}
\newtheorem{nexample}{\bf Example}

\newenvironment{theorem}
{\begin{nntheorem}\it}
{\end{nntheorem}}
\newenvironment{proposition}
{\begin{nnproposition}\it}
{\end{nnproposition}}
\newenvironment{lemma}
{\begin{nnlemma}\it}
{\end{nnlemma}}

\newenvironment{assumption}
{\begin{nnassumption}\it}
{\end{nnassumption}}

\newenvironment{nnexample}
{\begin{nexample}\rm}{\end{nexample}}

\newtheorem{nnremark}{\bf Remark}
\newenvironment{remark}{\begin{nnremark}}{\hfill \hspace*{1pt}\hfill $\circ$\end{nnremark}}

\newenvironment{proof}{{\bf Proof.}}{\hfill \hspace*{1pt}\hfill $\bullet$}

\def\RR{\mathbb{R}}
\def\NN{\mathbb{N}}

\def\epsilon{\varepsilon}

\def\LR{\mathcal L}     

\def\LR{{\mathfrak L}}  

\DeclareMathOperator{\argmax}{arg max}

\renewcommand{\leq}{\leqslant}
\renewcommand{\geq}{\geqslant}

\newcommand{\A}{\mathcal{A}}
\newcommand{\B}{\mathcal{B}}
\newcommand{\C}{\mathcal{C}}
\newcommand{\Hi}{\mathcal{H}}
\newcommand{\W}{\mathcal{W}}
\newcommand{\Pl}{\mathcal{P}}
\newcommand{\I}{\mathrm{I}}
\newcommand{\M}{\mathcal{M}}

\newcommand{\F}{\mathcal{F}}

\title{Forwarding-Lyapunov design for the stabilization of coupled ODEs and exponentially stable PDEs}

\author{Swann Marx$^{1}$,  Daniele Astolfi$^{2}$ and Vincent Andrieu$^{2}$
\thanks{$^{1}$ Swann Marx is with LS2N, \'Ecole Centrale de Nantes $\&$ CNRS UMR 6004, F-44000 Nantes, France.
{\tt\small swann.marx@ls2n.fr}.}
\thanks{$^{2}$ Daniele Astolfi and Vincent Andrieu 
are with Univ Lyon, Universit\'e Claude Bernard Lyon 1, CNRS, LAGEPP UMR 5007, 43 boulevard du 11 novembre 1918, F-69100, Villeurbanne, France.  {\tt\small name.surname@univ-lyon1.fr}} \thanks{This research was partially supported by the French Grant ANR ODISSE (ANR-19-CE48-0004-01).}}

\begin{document}

\maketitle

\begin{abstract}
This paper is about the stabilization of a cascade system composed by an infinite-dimensional system, that we suppose to be exponentially stable, and an ordinary differential equation (ODE), that we suppose to be marginally stable. The system is controlled through the infinite-dimensional system. Such a structure is particularly useful when applying the internal model approach on infinite-dimensional systems. Our strategy relies on the forwarding method, which uses a Lyapunov functional and a Sylvester equation to build a feedback-law. Under some classical assumptions in the output regulation theory, we prove that the closed-loop system is globally exponentially stable.  
\end{abstract}

\section{Introduction}

Many researchers devoted their attention 
in the last decades to the stabilization of
coupled systems composed by the interconnection or the cascade of ODEs and PDEs, see, e.g. \cite{marx2021forwarding,bribiesca,tang2011state}.
In this article, we consider a stabilization problem of cascade systems in which the first subsystem is an infinite-dimensional system, and the second one is an ODE.
This type of interconnections may appear, for instance, in two contexts: 
in output regulation problems, see, e.g.,  
\cite{xu1995robust,terrand2019adding,paunonen2014internal,paunonen2015controller,deutscher2015backstepping,deutscher2019robust},
or when the actuator dynamics is 
modelled as a PDE, 
see, e.g. \cite{krstic2009compensating,krstic2008backstepping,jankovic2009forwarding}.

In the context of output regulation, 
the typical internal model approach 
consists in adding in the loop of the controller an additional dynamics that is a copy 
(i.e. the internal model) of the dynamics 
generating the disturbances to be rejected and/or the references to be tracked. This external dynamics is usually denoted as exosystem. A really well known example of the internal model approach is the integral action controller  
for tracking or rejection 
of constant signals, see, e.g. 
\cite{terrand2019adding,balogoun2021iss,zhang2019pi,coron2019pi}. For more sophisticated exosystems
represented, for instance, by the combination of a finite number of linear oscillators, 
there exist some results devoted to abstract systems, i.e., systems described by operators, 
see, for instance, 
\cite{paunonen2014internal,paunonen2015controller}. These results are based mainly on a frequency-domain approach, and the systems that this theory targets are linear. We may also mention \cite{deutscher2019robust}, where the backstepping method is used in order to achieve some output regulation objectives. In this paper, we propose an alternative approach, based on the forwarding method. 

The forwarding method is a nonlinear Lyapunov method used for the stabilization of 
cascade systems, see, e.g. 
\cite{mazenc1996adding}. It is based on the construction of a Lyapunov function and of another function, which is given by, in the linear case, the solution of a Sylvester equation. 
To the best of authors' knowledge, until now, there exist few extensions of this method to the infinite-dimensional case:
\cite{jankovic2009forwarding}
for the stabilization of systems with 
input delays; 
\cite{terrand2019adding,balogoun2021iss}, where some PI controllers are designed for hyperbolic systems and a Korteweg-de Vries equation; \cite{marx2021forwarding,astolfi2021repetitive}, where a stabilization problem involving an infinite-dimensional system is considered. It is worth noticing that our approach might be used also in the context of finite-dimensional systems with an actuator governed by an infinite-dimensional system, as illustrated in \cite{bekiaris2014compensation}
or \cite{jankovic2009forwarding}.

Our contribution is therefore the following: under some structural assumptions, by means of the forwarding method, we design a feedback-law in order to stabilize a cascade system composed by an ODE and an infinite-dimensional system described by an abstract operator. The control, which appears in the infinite-dimensional system, is described with a bounded operator, which corresponds to a distributed control in the PDE context. The output, modeled with an unbounded operator, i.e. the trace of the solution in the PDE case, enters the ODE. The abstract operator is supposed
to be exponentially stable (or, alternatively, that 
has already been stabilized) while the ODE is supposed to be dissipative.
The proposed design uses a linear partial-state feedback law which uses only the state of the ODE, and obtained via the solution of a Sylvester equation. We further are able to deduce that the origin of the closed-loop system is exponentially stable by using a strictification technique presented in \cite{praly2019observers}. Finally, we obtain a condition on a scalar gain allowing to achieve the stabilization of the cascade system.

We believe that our design may be, in some cases, easier to apply than  the design proposed in 
\cite{paunonen2015controller}
as we do not need the computation 
of the transfer function for abstract operators.
Furthermore, we also generalize the construction in 
\cite{paunonen2015controller}
as we allow the ODE to be dissipative 
and not only conservative. 
We also extend the class of abstract operator
with respect to the technique proposed in 
\cite{deutscher2019robust} based on the backstepping approach.
Finally, since the proof is based on the 
use of a
Lyapunov functional (and not on purely linear arguments), we 
believe that the proposed
result may be extended in the future to some classes of 
nonlinear systems.

This rest of the paper is organized as follows. Section~\ref{sec_main} presents the mathematical context of this problem and the main results. In Section~\ref{sec_proofs}, proofs of our main results are given. Section~\ref{sec_illustration} illustrates our results on an example, namely a heat equation. Finally, Section~\ref{sec_conclusion} we draw the conclusions and we gather some further research lines to be followed.

\textbf{Notation.} Set $\RR_+=[0,\infty)$. For any $n\in \NN$, we denote by $|\cdot|$ the Euclidean norm in $\RR^n$ and, with a slight abuse of notation, we keep the same notation for matrix norm. Given two Hilbert spaces $\Hi_1$ and $\Hi_2$, the space $\LR(\Hi_1,\Hi_2)$ denotes the space of functions bounded from $\Hi_1$ to $\Hi_2$, and $\LR(\Hi_1)=\LR(\Hi_1,\Hi_1)$. Given a Hilbert space $\Hi$, $\I_\Hi$ denotes the identity operator. For $\RR^n$, this identity operator is given by $\I_n$. 

\section{Main results}

\label{sec_main}
\subsection{Problem statement}
Consider a Hilbert space $\Hi$ equipped with the norm $\Vert \cdot\Vert_\Hi$ and the scalar product $\langle \cdot,\cdot\rangle_\Hi$. We are interested in the following cascade system:

\begin{equation}
\label{eq:open-loop}
\left\{
\begin{array}{ll}
\frac{d}{dt} \phi(t) = \A \phi(t) + \B u(t),\\[.5em]
\frac{d}{dt}z(t) = Sz(t) + \Gamma \C \phi(t),\\[.5em]
\phi(0)=\phi_0, \: z(0) = z_0,
\end{array}
\right.
\end{equation}
 where  $\A:\: D(\A)\subset \Hi\rightarrow \Hi$ with $D(\A)$ densely defined in $\Hi$, $\B\in \LR(\RR^m,\Hi)$, $\C \in \LR(D(\A),\RR^p)$, $S\in \RR^{r\times r}$, and $\Gamma\in \RR^{r\times p}$. Therefore, the control operator $\B$ is supposed to be bounded, while the output operator might be unbounded. 

The objective of this work is to design a feedback law for the stabilization of the origin of system \eqref{eq:open-loop}. To this end, we follow the so-called  ``forwarding approach'' introduced
for finite dimensional nonlinear systems in \cite{mazenc1996adding}, and already successfully 
extended to other context of cascade systems composed by one PDE and one ODE, see, e.g., 
\cite{terrand2019adding}, \cite{marx2021forwarding}.
 To do so, the following set of assumptions is stated.

\begin{assumption}
\label{ass:openloopstable}
 The operator $\A$ generates a strongly continuous semigroup of contractions. Moreover, there exist a positive value $\mu$ and a self-adjoint, positive and coercive operator $\Pl\in\LR(\Hi)$ such that, for every $\phi\in~ D(\A)$
    \begin{equation}
    \label{eq:lyap-ineq}
        \langle \Pl \A \phi,\phi\rangle_{\Hi} + \langle \Pl \phi,\A \phi\rangle_{\Hi}  \leq - \mu \Vert \phi\Vert^2_{\Hi}.
    \end{equation}
\end{assumption}

\begin{assumption}
\label{ass:marginallystable}
The pair 
     $(S,\Gamma)$ is controllable. Moreover there exists a symmetric positive definite matrix $Q$ such that 
     $QS+ S^\top Q \leq 0$.   
\end{assumption}

\begin{assumption}\label{ass:disjoint}
 The spectra of $\A$ and $S$ are disjoint and non-empty.
\end{assumption}

Following the forwarding paradigm \cite{mazenc1996adding}, we ask 
the $\phi$-dynamics to be ``open-loop exponentially stable'', see Assumption~\ref{ass:openloopstable}, and the $z$-dynamics to be marginally stable, see
Assumption~\ref{ass:marginallystable}. It is also worth saying that the coercivity assumption written in Assumption~\ref{ass:openloopstable} on $\Pl$ is due to the infinite-dimensional aspect of our problem. Indeed, while such a property is naturally satisfied for ODEs, it is no longer the case for infinite-dimensional systems, as illustrated in \cite{mironchenko2019non}. This assumption is instrumental to show the well-posedness of the closed loop system.
Assumption~\ref{ass:disjoint} is needed to apply the forwarding approach, since it ensures the existence of the feedback-law. Note that when the matrix $S$ has only eigenvalues lying on the imaginary axes (in this case $QS+S^\top Q= 0$) then it is trivially satisfied.

\subsection{Control Design and Well-posedness}
In order to develop the forwarding methodology 
for the class of systems \eqref{eq:open-loop}, 
we first introduce the operator $\M:~D(\A) \rightarrow \RR^r$, defined as the solution of the following \emph{Sylvester equation}
\begin{equation}
\label{eq:Sylvester}
S\M - \M\A = - \Gamma \C.
\end{equation}
Using Assumptions~\ref{ass:openloopstable}, 
\ref{ass:marginallystable} and \ref{ass:disjoint}, and invoking \cite[Lemma 22]{Phong1991}, existence and uniqueness of the solution of the previous equation is guaranteed. 
In the following result, we extend the domain of the operator $\M$ from $D(\A)$ to $\Hi$. With a slight abuse of notation, 
in the rest of the article, 
we denote such an extension as $\M$. We also show that $\M$ can be identified uniquely with a vector of functions in $\Hi$ through the scalar product $\langle \cdot,\cdot\rangle_\Hi$. 

\begin{lemma}
\label{lemma:M}
  The solution $\M$ to \eqref{eq:Sylvester} can be extended to an operator, still denoted by $\M$, with domain $\Hi$. Moreover, for every $\phi\in\Hi$, $\M$ can be uniquely defined as follows:  
$$
\M \phi = \begin{bmatrix}
\langle M_1,\phi\rangle_\Hi & \ldots & \langle M_r,\phi\rangle_\Hi
\end{bmatrix}^\top,
$$
where $M_i\in\Hi$, $i\in \lbrace 1,\ldots,r\rbrace$.
\end{lemma}

\begin{proof} One can rewrite \eqref{eq:Sylvester} as
\begin{equation*}
\M = \Gamma\C\A^{-1} + S\M \A^{-1}.
\end{equation*}
Noting that the operator $\A^{-1}$ is in $\LR(\Hi,D(\A))$, one can show that $\M$ can be extended as $\overbar\M :\Hi \rightarrow \RR^r$. In the sequel, we will make the abuse of notation $\overbar \M=\M$. 

The operator $\M$ is therefore a continuous operator since, for every $\phi\in \Hi$, one has
\begin{equation}\label{eq:M_norm}
|\M\phi|\leq \Vert \M\Vert_{\LR(\Hi,\RR^r)} \Vert \phi\Vert_{\Hi}.
\end{equation}
Noticing that, for every $\phi\in \Hi$, $\M\phi\in\mathbb{R}^r$, one can write $\M\phi:=\begin{bmatrix}
(\M\phi)_1 & \ldots & (\M\phi)_r
\end{bmatrix}^\top$. Hence, each operator $\M_i:\phi\in\Hi\mapsto (\M\phi)_i\in\RR$, $i\in\lbrace 1,\ldots,r\rbrace$, is a continuous linear form on $\Hi$. Using the Riesz representation theorem \cite[Theorem 4.12.]{walter1987real}, the proof concludes.
\end{proof}

Given the operator $\M$ defined in \eqref{eq:Sylvester}, we can now design 
the feedback-law for the system \eqref{eq:open-loop}. Furthermore, in order 
to provide a feedback law easy to implement, we look for an output-feedback design that uses only the $z$ variable but not $\phi$. In particular, it is given by 
\begin{equation}
\label{eq:feedback-law}
u(t):= kH Q z(t),
\end{equation}
where $Q$ is given by Assumption~\ref{ass:marginallystable}, $k$ is a positive constant to be defined later on, and $H$, a matrix of appropriate dimension, is defined as
\begin{equation}
\label{eq:H}
H:=\B^*\M^*.
\end{equation}
Denoting the state by $w:=\begin{bmatrix} z & \phi \end{bmatrix}^\top$ and the state space by $\W:=\RR^r \times \Hi$, the closed-loop system therefore reads:
\begin{equation}
\label{eq:closed-loop}
\left\{
\begin{array}{ll}
\frac{d}{dt} w(t) = \F w(t), \\[.5em]
w(0) = w_0,
\end{array}
\right.
\end{equation}
where the operator $\F:D(\F)\subset \W\rightarrow \W$ is defined as
\begin{equation}
\F:=\begin{bmatrix}
\A & k\B H Q\\
\Gamma \C & S
\end{bmatrix},
\end{equation}
and  $D(\F):=D(\A)\times \mathbb{R}^r$, since the operator $\B$ is bounded. We are now in position to state our main results.

\begin{theorem}[Well-posedness]
\label{thm:wp}
Suppose Assumptions~\ref{ass:openloopstable}, \ref{ass:marginallystable} and \ref{ass:disjoint}
hold. Let
 us define $k^*$ as 
\begin{equation}
    \label{eq:kstar}
    k^* := \rho \,
    \max_{s\in(0,1)}\sqrt{\dfrac{1-s}{1+s} \, s}
\end{equation}
with 
\begin{equation}
    \label{eq:rho}
\rho := \dfrac{\mu}{\Vert \Pl\Vert_{\LR(\Hi)} \Vert \B\Vert_{\LR(\RR^m,\Hi)} \Vert \M\Vert_{\LR(\Hi,\mathbb{R}^r)} |H|} .
\end{equation}
For every $k\in (0,k^*)$ and for every initial conditions $(z_0,\phi_0)\in \W$ (resp. $(z_0,\phi_0)\in D(\F))$, there exists a unique solution $(z,w)\in C(\RR_+;\W)$ (resp. $(z_0,\phi_0)\in~ C^1(\RR_+;D(\F))$) to \eqref{eq:closed-loop}.
\end{theorem}

Theorem~\ref{thm:wp} ensures  
the well posedness
of the solutions of the system 
\eqref{eq:open-loop} in closed-loop 
with the feedback law \eqref{eq:feedback-law}, 
provided that the parameter 
$k$ in the feedback law \eqref{eq:feedback-law}
is selected small enough.

\subsection{Sufficient condition for exponential stability}

In order to show the stability of the origin of 
the closed-loop system
\eqref{eq:closed-loop},
the following extra assumption is needed. 

\begin{assumption}
\label{ass:observability}
The pair $(S,HQ)$ is detectable\footnote{In other words, there exists $L$ such that 
the matrix $(S-LHQ)$ is Hurwitz.},
with $Q$ given by Assumption~\ref{ass:marginallystable} and $H$ defined as in \eqref{eq:H}.
\end{assumption}

A detailed discussion about Assumption~\ref{ass:observability} is given in 
the next section, after establishing the second main result of this article.

\begin{theorem}[Global exponential stability]
\label{thm:as}
Suppose Assumptions~\ref{ass:openloopstable}, \ref{ass:marginallystable}, \ref{ass:disjoint} and \ref{ass:observability} hold. Then with $k^*$ defined in \eqref{eq:kstar}, for every $k\in (0,k^*)$ and for every initial conditions, $(z_0,\phi_0)\in W$, the equilibrium point $0$ of \eqref{eq:closed-loop} is globally exponentially stable.
\end{theorem}

Theorem~\ref{thm:as} states that 
the origin of the closed-loop system 
\eqref{eq:open-loop}, \eqref{eq:feedback-law} is exponentially stable, provided
the detectability condition of Assumption~\ref{ass:observability} holds.
The motivation to such an assumption is that 
such a condition is sufficient 
to build a strict-Lyapunov function as shown in its proof in Section~\ref{sec_proof_2}.

\subsection{About Assumption~\ref{ass:observability}}

Assumption~\ref{ass:observability} can be directly checked provided
that one is able to solve the Sylvester equation \eqref{eq:Sylvester}, 
whose solution is anyway needed in order to design the feedback law
\eqref{eq:feedback-law}. Although such an observability condition is 
directly related to the data of the problem (i.e. the operators 
$\A,\B,\C$ and the matrices $S,\Gamma$) via the Sylvester equation \eqref{eq:Sylvester}, 
one may ask whether alternative conditions can be established.
A partial answer is given in the case in which the matrix $S$ is 
skew-symmetric, i.e. when $Q$ in Assumption~\ref{ass:marginallystable}
coincides with the identity matrix.
In such a case, we have the following result.

\begin{proposition}
\label{prop:observability}
Suppose that $S$ is a skew-symmetric matrix and that the following holds
    \begin{equation}\label{eq:nonres}
        \mathrm{Ran}\begin{bmatrix}
        \A - \I_\Hi \lambda & \B\\
        \C & 0
        \end{bmatrix}= \Hi\times \RR^p
    \end{equation}
    for any $\lambda$ eigenvalue of $S$.
    Then, the pair $(S,H)$ with $H$ defined in \eqref{eq:H} is observable.
\end{proposition}

\begin{proof}
Pick $\varphi$ an eigenvector of $S$ and denote $-\lambda$ its corresponding eigenvalue. Recalling that $H=\mathcal{B}^*\mathcal{M}^*$, we suppose that $\mathcal{B}^*\mathcal{M}^*\varphi=0$ and look for a contradiction. Since $S$ is skew-symmetric, $\varphi$ is also an eigenvector of $S^\top$, with its corresponding eigenvalue $\lambda$. Then, according to \eqref{eq:Sylvester}, one has
\begin{equation*}
    \varphi^\top S\M - \varphi^\top \M\A = -\varphi^\top \Gamma \C,
\end{equation*}
that is
\begin{equation*}
    \varphi^\top \M(\lambda \I_\Hi- \A) + \varphi^\top \Gamma \C = 0.
\end{equation*}
This implies that:
\begin{equation}
    \begin{bmatrix}
    \varphi^\top \M & \varphi^\top \Gamma
    \end{bmatrix}\begin{bmatrix}
    \lambda \I_{\Hi} - \A & \B\\
    \mathcal{C} & 0
    \end{bmatrix}=0
\end{equation}
Due to the condition \eqref{eq:nonres} one has that $\Gamma^\top \varphi=0$, which is in contradiction with
the fact that the pair $(S,\Gamma)$ is controllable, 
see Assumption~\ref{ass:marginallystable}. This concludes the proof of the proposition.
\end{proof}

 The condition \eqref{eq:nonres} is known also in output regulation theory  as the \emph{non-resonance condition}. 
 In the context of 
 finite-dimensional linear systems (i.e. when the operators $\A,\B,\C$ are matrices), such a condition is also shown to be necessary and sufficient
  for the controllability of the cascade system  \eqref{eq:open-loop} together with the controllability of the pair $(S,\Gamma)$, see
\cite{davison1975new}. Note   that  
the condition 
\eqref{eq:nonres}
implies also
that $m\geq p$, i.e., the number of control inputs $u$ is larger or equal the the number of outputs $\C \phi$.

As a consequence, when $S$ is a skew-symmetric matrix, 
the Assumption~\ref{ass:observability} in Theorem~\ref{thm:as}
can be replaced by the condition \eqref{eq:nonres}.
We recover then the invertibility condition in Theorem 5.2 and 5.3 of \cite{paunonen2015controller}, 
or the controllability condition 
of Theorem 1 
in \cite{deutscher2019robust} 
for the particular case of 
parabolic partial
integro-differential equations.
In this case, 
with respect to \cite{paunonen2015controller},
the interest of 
the proof of Theorem~\ref{thm:as} is that we provide a
Lyapunov functional for the closed-loop system and a design
of a feedback which is not based
on the computation of the transfer function of the infinite-dimensional system 
$(\A,\B,\C)$, i.e. the computation of 
$\C(\lambda \I_\Hi - \A)^{-1}\B$.
With respect to \cite{deutscher2019robust}, 
we generalize the classes of PDEs 
represented by the operator $\A$, see, the 
example at the end given below in Section~\ref{sec_illustration}
or the special case of integral action control
in  Korteweg-de Vries  equation \cite{balogoun2021iss}.

Finally, it is worth to highlight that in practice, verifying
the condition \eqref{eq:nonres} can be harder than
verifying the detectability condition of Assumption~\ref{ass:observability}, 
see, for instance, \cite{terrand2019adding,deutscher2019robust} and also 
the example given below in Section~\ref{sec_illustration}.

\section{Proofs of the main results}
\label{sec_proofs}

\subsection{A preliminary result}

The proofs of our main results rely on Lyapunov arguments. For this, we introduce  the following Lyapunov functional
\begin{equation}
\label{eq:lyapunov}
V(z,\phi):=\langle \Pl\phi,\phi\rangle_{\Hi} + p (z-\M \phi)^\top Q (z - \M \phi)
\end{equation}
where $p>0$ is a positive constant to be defined later on. 
Associated to this Lyapunov functional, one can define a scalar product, 
defined as 
\begin{equation}
\langle w_1,w_2\rangle_V:= \langle \Pl\phi_1,\phi_2\rangle_{\Hi} + p (z_1-\M\phi_1)^\top Q (z_2-\M\phi_2) 
\end{equation}
for any $w_1,w_2\in W$.
It is easy to verify  that $V(z,\phi)=\Vert w\Vert^2_V$. 
Before providing  the proof of our main results, let us prove that the square root of this Lyapunov functional (resp. this scalar product) is equivalent to the usual norm (resp. the usual scalar product).

\begin{lemma}\label{lemma:norm}
The square root of $V$ and $\langle \cdot,\cdot\rangle_V$ are equivalent to the usual norm and the usual scalar product in $\RR^r\times \Hi$, respectively.
\end{lemma}

\begin{proof}
Proving that $\sqrt{V}$ is equivalent to the usual norm in $\RR^r \times \Hi$ is sufficient to prove that $\langle \cdot,\cdot\rangle_V$ is equivalent to the usual scalar product. 

First, note that, using the Cauchy-Schwarz inequality
\begin{multline}
\label{eq:equivalence1}
V(z,\phi)\leq  \Vert \Pl\Vert_{\LR(\Hi)} \Vert \phi\Vert^2 +p q |z|^2 
\\ 
 \quad +p q\Vert \M\Vert^2_{\LR(D(\A),\RR^r)} \Vert \phi\Vert^2_{\Hi},
\end{multline}
where $q:= |Q|$.
Second, since $\Pl$ is coercive, there exists $\alpha>0$ such that $\langle \Pl\phi,\phi\rangle_{\Hi}\geq \alpha \Vert \phi\Vert^2_{\Hi}$. Note, moreover, that one can show that the following inequality 
$$
|v_1-v_2|^2 \geq \theta\left(\frac{1}{2} |v_1|^2 - |v_2|^2\right)
$$
holds for every $\theta\in (0,1)$, and for every $v_1,v_2\in\mathbb{R}^r$.
These properties together yield the following
\begin{align}
V(z,\phi)\geq &\alpha \Vert \phi\Vert_{\Hi}^2 + \theta q \left(p|z|^2 - p\Vert \M\Vert_{\LR(D(\A),\RR^r)} \Vert \phi\Vert^2_{\Hi}\right)\notag\\
\geq &(\alpha - qp\theta \Vert \M\Vert_{\LR(D(\A),\RR^r)})\Vert \phi\Vert_{\Hi} + qp\theta |z|^2. \label{eq:equivalence2}
\end{align}
Equations \eqref{eq:equivalence1} and \eqref{eq:equivalence2} with $\theta$ sufficiently small are sufficient to prove that $\sqrt{V}$ is equivalent to the usual norm, completing  the proof of this lemma.\end{proof}

\subsection{Proof of Theorem \ref{thm:wp}}
\label{sec_proof_1}

This subsection is devoted to the prove Theorem \ref{thm:wp}. Our strategy relies on semigroup arguments.
In particular, to prove this well-posedness result, we will show that $\F$ generates a strongly continuous semigroup of contractions. Then, applying \cite[Theorem 4.3]{pazy1983semigroups}, one can deduce the statement of Theorem \ref{thm:wp} invoking \cite[Theorems 1.3. $\&$ 1.4.]{pazy1983semigroups}. Showing this needs to prove that the operator $\F$ is dissipative and  maximal.

\paragraph{Dissipativity of $\F$} For any $w\in D(\F)$, one has
\begin{align*}
\langle \F w,w\rangle_V= &\langle \Pl(\A\phi+k\B HQz),\phi\rangle_{\Hi} \\
&+ \langle \Pl \phi,\A\phi + k\B H z\rangle_{\Hi}\\
&+ 2p (z-\M\phi)^\top Q
(Sz+ \Gamma \C \phi -\M \A \phi) \\
& -  2p(z-\M\phi)^\top Q (k\M\B H Q z).
\end{align*}
Using the Sylvester equation \eqref{eq:Sylvester}, the Lyapunov inequality given in \eqref{eq:lyap-ineq} and the fact that $S$ 
satisfies Assumption~\ref{ass:marginallystable},
the third term of the previous equation satisfy 
\begin{align*}
     2p(z-\M\phi)^\top Q 
     (Sz+ \Gamma \C \phi -\M \A \phi)  = &
    \\
 p (z - \M \phi)^\top ( Q S + S^\top Q) (z-\M\phi) 
& \leq 0.
\end{align*}
As a consequence, 
one obtains, for any $w\in D(\F)$,
\begin{align*}
&\langle \F w,w\rangle_V 
\\
&\qquad \leq -\mu \Vert \phi\Vert^2_\Hi + 2\langle k\Pl\B HQ z,\phi\rangle_{\Hi} \\
&\qquad \qquad- 2p(z-\M\phi)^\top Q (k\M\B H Q z)\\
&\qquad \leq  - \mu \Vert \phi\Vert^2_\Hi - 2pk|HQz|^2\\
&\qquad\qquad + 2k\langle \Pl\B HQz,\phi\rangle_{\Hi} + 2pk z^\top QH^\top H \M\phi,
\end{align*}
where, in the last line, we have used the definition of $H$ given in \eqref{eq:H}.
Then, using the Young's inequality, one obtains, for every $\nu>0$
\begin{multline*}
\langle \Pl \B HQz,\phi\rangle_{\Hi} \leq \\ \nu \Vert \phi\Vert^2_{\Hi} 
 + \dfrac1{\nu} \Vert \Pl\Vert^2_{\LR(\Hi)} \Vert \B\Vert^2_{\LR(\RR^m,\Hi)}|HQz|^2
\end{multline*}
and 
\begin{multline*}
2z^\top Q H^\top H \M\phi \leq \\  \Vert \M\Vert^2_{\LR(\Hi,\mathbb{R}^r)} |H|^2\Vert \phi\Vert^2_{\Hi} + |HQz|^2.
\end{multline*}
Combining all the equations together, one has, for all $w\in D(\F)$
\begin{align}\notag
&\langle \F w,w\rangle_V \leq -\left(\mu-{pk} \Vert \M\Vert^2_{\LR(\Hi,\mathbb{R}^r)} |H|^2 - {k}\nu\right) \Vert \phi\Vert_{\Hi}^2\\
& \quad - k\left(p - \dfrac{1}\nu \Vert \Pl\Vert^2_{\LR(\Hi)} \Vert \B\Vert^2_{\LR(\RR^m,\Hi)}\right) |HQz|^2.
\label{eq:Lyap_1}
\end{align}
Then, given any $\varepsilon\in(0,1)$,
select $\nu$ and $p$ as
$$
\nu = \mu\dfrac{1-\varepsilon}k,
\qquad p = \dfrac{1}{\nu} \Vert \Pl\Vert^2_{\LR(\Hi)} \Vert \B\Vert^2_{\LR(\RR^m,\Hi)} (1+\varepsilon)
$$
Inequality \eqref{eq:Lyap_1}
gives
\begin{multline}
\langle \F w,w\rangle_V \leq
-\mu\left(\varepsilon-  \dfrac{1+\varepsilon}{1-\varepsilon} \dfrac{k^2}{\rho^2} \right) \Vert \phi\Vert_{\Hi}^2.
\\ 
- \dfrac{\varepsilon}{1+\varepsilon} kp  |HQz|^2.
\label{eq:Lyap_2}
\end{multline}
with $\rho$ defined as in \eqref{eq:rho}.
Then, 
define 
\begin{equation*}
\varepsilon^* = 
\underset{s\in(0,1)}{\argmax} \;
 \sqrt{\dfrac{s(1-s)}{1+s}} 
\end{equation*}
and select $\varepsilon= \varepsilon^*$.
By definition of $k^*$
given in the statement of the theorem, 
we conclude, 
from inequality \eqref{eq:Lyap_2}, 
that
for any $k\in (0,k^\star)$, 
there exists positive constants $a,b>0$ such that the following 
inequality
\begin{align}
\label{eq:lyapunov-ineq1}
\langle \F w,w\rangle_{V} \leq -a\Vert \phi\Vert^2_{\Hi} - b|Hz|^2,
\end{align}
holds for all  $w\in D(\F)$, 
thus showing that $\F$ is dissipative.

\paragraph{Maximality of $\F$} Proving that $\F$ is maximal consists in showing that, for a given $\lambda>0$, one has
$$
\W=\mathrm{Ran}(\F -\lambda \I_W).
$$
This reduces to proving that, for all $w\in \W$, there exists $\tilde{w}\in D(\F)$ such that
$$
\F \tilde{w} - \lambda \tilde{w} = w,
$$
which corresponds to the following problem: for all $(z,\phi)\in \W$, there exists a unique $(\tilde{z},\tilde{\phi)}\in D(\F)$
\begin{equation}
\label{eq:maximality}
\left\{
\begin{array}{ll}
\A\tilde{\phi} + \B HQ\tilde{z} - \lambda \tilde{\phi} = \phi,\\
S \tilde{z} + \Gamma \C \tilde{\phi} - \lambda \tilde{z} = z. 
\end{array}
\right.
\end{equation}
From the first line, using the fact that $\A$ generates a strongly continuous semigroup of contractions which is exponentially stable, one has
$$
\tilde{\phi} = (\A-\lambda \I_\Hi)^{-1}(\phi - \B HQ\tilde{z}),
$$
which implies that, once one has a solution $\tilde{z}$ depending only on $\phi$ and $z$, then one can deduce that there exists a solution $\tilde{\phi}$. The second line of \eqref{eq:maximality} together with the previous equation yields 
\begin{align*}
(S-\Gamma \C (\A - \lambda \I_\Hi)^{-1} \B HQ - &\lambda \I_r)\tilde z =  z \\
& - \Gamma \C(\A-\lambda \I_\Hi)^{-1}\phi,
\end{align*}
that may be rewritten as
\begin{align*}
(\I_r + &(S-\lambda \I_r)^{-1} \Gamma \C (\A - \lambda \I_\Hi)^{-1} \B H Q)\tilde z  = \\
&(S-\lambda \I_r)^{-1} (z-\Gamma \C(\A-\lambda \I_\Hi)^{-1}\phi)
\end{align*}
It remains to show that the matrix $\I_r + (S-\lambda \I_r)^{-1} \Gamma \C (\A - \lambda \I_\Hi)^{-1} \B H Q$ is invertible, and our result follows. To do so, it suffices to prove that there exists $\lambda>0$ such that $| (S-\lambda \I_r)^{-1} \Gamma \C (\A - \lambda \I_\Hi)^{-1} \B H Q|<1$. 

We select now $\lambda$ sufficiently large in order to prove our result. Indeed, if $\lambda > |S|$, then $|(\lambda \I_r - S)^{-1}|\leq \frac{1}{\lambda - |S|}$ (see, e.g., \cite[Lemma 2.2.6.]{tucsnak2009observation}). Therefore, $$|(\lambda \I_r - S)^{-1}|\to 0 \text{ as } \lambda\to +\infty.$$ Moreover, one has:
\begin{align*}
|\Gamma \C &(\A - \lambda \I_\Hi)^{-1} \B H Q|\leq \\
&\alpha_1 \alpha_2\Vert (\A - \lambda \I_\Hi)^{-1}\Vert_{\LR(\Hi,D(\A))},
\end{align*}
where $\alpha_1:=|\Gamma| \Vert \C\Vert_{\LR(D(\A),\mathbb{R}^p)}$ and $\alpha_2:=\Vert \B\Vert_{\LR(\mathbb{R}^m,\Hi)} |HQ|$.

Using \cite[Corollary 2.3.3.]{tucsnak2009observation}, one can show that the term $\Vert (\A - \lambda \I_\Hi)^{-1}\Vert_{\LR(\Hi,D(\A))}$ remains bounded, since $\mathcal{A}$ generates a strongly continuous semigroup of contractions which is exponentially stable. Then, finally, one can select $\lambda$ large enough so that to satisfy $$| (S-\lambda \I_r)^{-1} \Gamma \C (\A - \lambda \I_\Hi)^{-1} \B H Q|<1,$$ which achieves the proof.\qed

\begin{remark}
\label{rem:lyap}
As already noticed, there is a link between the scalar product $\langle \cdot,\cdot\rangle_V$ and the Lyapunov functional $V$. It is also worth noticing that, as soon as one considers strong solution, the time derivative of $V$ along solutions to \eqref{eq:closed-loop} satisfies $\frac{d}{dt} V(w) = \langle \F w,w\rangle_V$, 
which shows that, using \eqref{eq:lyapunov-ineq1},
the following inequality
\begin{equation}
\label{eq:lyapunov-ineq2}
\frac{d}{dt} V(z,\phi) \leq - a \Vert \phi\Vert^2_\mathcal{H} - b |HQz|^2,
\end{equation}
holds
for all $(z,\phi)\in D(\F)$. 
\end{remark}




\subsection{Proof of Theorem \ref{thm:as}}
\label{sec_proof_2}

In this proof, we will focus on strong solutions, i.e. for initial conditions $(z_0,w_0)$ in $D(\F)$. We can deduce the exponential stability result for weak solutions by using a density argument, as the one given in \cite[Lemma 1]{marx2017cone}.

Now, by looking at equation \eqref{eq:lyapunov-ineq2}, 
we can immediately note that the Lyapunov functional $V$ defined in 
\eqref{eq:lyapunov} is not ``strict'', 
in the sense that 
the right hand side of the inequality is  $|HQz|$, which is only a part of the state.
Indeed, in general, the rank of the matrix $H$ is strictly smaller than the dimension
of the state $z$. Note that an exception is the case in which $S=0$ and the dimensions 
of $u$ and $z$ coincide, as in the integral action control 
\cite{terrand2019adding}. In such a case, the Lyapunov function $V$
is sufficient to establish the result of Theorem~\ref{thm:as}.
As a consequence, in order to establish the exponential stability of the
closed-loop system, 
the purpose of this section is to construct
a strict Lyapunov function based
on the detectability properties
of the pair $(S,HQ)$ stated in Assumption~\ref{ass:observability}, by 
strictifying\footnote{Such a technique is inspired by the
use of observers to build strict Lyapunov functions
introduced in \cite{praly2019observers}
and successfully applied also in 
\cite{balogoun2021iss} in the context of a a Korteweg-de Vries equation.
In our particular context, it can be also obtained as a variation  of the Lemma~1 in \cite{astolfi2021nonlinear}.} the Lyapunov function $V$.
In particular,  since the pair $(S,HQ)$
is detectable, there exist
a symmetric positive definite matrix $\Pi$
and  a matrix $L$ such that 
\begin{equation}
    \label{eq:observerLyap}
    \Pi(S- L HQ) +(S- L HQ)^\top \Pi \leq - 2\I_r.
\end{equation}
Then, consider the Lyapunov functional 
$W$ defined as 
\begin{align}
\label{eq:lyapunov2}
\begin{split}
W(z,\phi) &:= V(z,\phi) + cU(z,\phi) ,\\
U(z,\phi) & :=  (z-\M \phi)^\top \Pi (z - \M \phi),
 \end{split}
\end{align}
with $V$ defined as in \eqref{eq:lyapunov}
and  $c>0$ being a small coefficient to be 
selected.
Note that following the computations of Lemma~\ref{lemma:norm}, 
it is straightforward to show that the Lyapunov function 
$W$ defined 
in \eqref{eq:lyapunov2} 
is equivalent to the usual norm 
and the usual scalar product in $\RR^r\times \Hi$ (this can be shown by using 
similar arguments to those employed in Lemma~\ref{lemma:norm}).

Then, we compute 
the time-derivative of $U$ 
along solutions to 
\eqref{eq:closed-loop}.
Using the definition of $\M$ given by \eqref{eq:Sylvester}, 
one obtains
\begin{align*}
\frac{d}{dt} U(z,\phi) 
 = &2(z-\M \phi)^\top \Pi (S- LHQ)(z- \M \phi)
\\
& -2(z-\M \phi)^\top \Pi (k\M \B - L)HQz
\\ & +2(z-\M \phi)^\top \Pi HQ\M \phi.
\end{align*}
Using the Young inequality
and \eqref{eq:M_norm}, the 
last two terms can be bounded as follows
\begin{multline*}
    2(z-\M \phi)^\top \Pi HQ\M \phi
\leq \\
2|\Pi HQ|^2   \Vert \M\Vert_{\LR(\Hi,\RR^r)}^2     \Vert \phi\Vert_{\Hi}^2 + 
 \dfrac{1}2|z-\M \phi|^2
\end{multline*}
and
\begin{multline*}
-2(z-\M \phi)^\top \Pi (k\M \B - L)HQz
\leq \\
2|\Pi (k\M \B - L)|^2 |HQz|^2  +  \dfrac{1}2|z-\M \phi|^2.
\end{multline*}
Therefore, using 
\eqref{eq:observerLyap}, 
we further obtain
\begin{align}\label{eq:Uder}
\frac{d}{dt} U(z,\phi)  \leq  & -  |z-\M \phi|^2
  +\nu_1 |\phi|^2  + \nu_2|HQz|^2 ,
\end{align}
with $\nu_1,\nu_2$ defined as
\begin{align*}
\nu_1 &= 2|\Pi HQ|^2  \Vert \M\Vert_{\LR(\Hi,\RR^r)}^2 ,
\\
\nu_2 &= 2|\Pi (k\M B - L)|^2 |HQz|^2.    
\end{align*}
Combining 
\eqref{eq:Uder} with 
\eqref{eq:lyapunov-ineq2}, 
the time-derivative of $W$
can be finally computed as
\begin{multline*}
\frac{d}{dt} W(z,\phi)  \leq -(a- c\nu_1)   \Vert \phi\Vert_{\Hi}^2  - 
(b- c\nu_2)|HQz|^2  \\ - c|z-\M \phi|^2 .
\end{multline*}
By selecting
$$
c < \min\left\{ \dfrac{a}{\nu_1}, \;
\dfrac{b}{\nu_2}\right\},
$$
we finally obtain
\begin{equation}
    \label{eq:Wderiv}
\frac{d}{dt} W(z,\phi)  \leq -\varepsilon  \left( \Vert \phi\Vert_{\Hi}^2  + c|z-\M \phi|^2 \right)
\end{equation}
for some $\varepsilon>0$,
showing the exponential stability 
of the origin.\qed

\section{An illustration on the heat equation}
\label{sec_illustration}


\subsection{The system under consideration}
We consider, as illustration, the case 
of a linear system controlled via an actuator with dynamics
 described by a heat equation  \cite{krstic2009compensating}.
In particular, for a positive integer $r$ and $(\phi_0,z_0)$ in $L^2(0,1)\times\RR^r$, we consider the following  system
\begin{equation}\label{eq_chaleur}
\left\{
    \begin{array}{ll}
        &\begin{multlined}
        \phi_t(t,x)  =  \phi_{xx}(t,x) + b(x) u(t), \\[-1.5em] (t,x)\in \RR_+\times(0,1),
        \end{multlined}
        \\
        &z_t(t) = S z(t) + \Gamma \phi(t,\ell),\quad t\in\RR_+,
      \\[.5em]
        &\phi(t,0)  = \phi(t,1) = 0,\quad t\in \RR_+,\\
        &(\phi(0,x),z(0))=(\phi_0(x),z_0) ,\ x\in (0,1).\ 
    \end{array}
\right.
\end{equation}
where $b$ is in $L^2(0,1)$ and $\ell$ in $(0,1)$
and with $S=-S^\top$ in $\RR^{r\times r}$ and $\Gamma$ in $\RR^r$ such that the pair $(S,\Gamma)$ is controlable.

This system can be written in the form \eqref{eq:open-loop} by setting $\Hi:=L^2(0,1)$ and
\begin{equation}
    \A \phi:=\phi^{\prime\prime},\qquad \B u:=bu,\qquad \C \phi:= \phi(\ell),
\end{equation}
with $D(\A):=\lbrace \phi\in H^2(0,1)\mid \phi(0)=\phi(1)=0\rbrace$. First note that Assumption \ref{ass:openloopstable} is satisfied. Indeed, the operator $\Pl$ is given by $\I_\Hi$, and $\mu$ is given by $\pi$. Indeed, with some integration by parts, it yields for every  $\phi$ in $D(\A)$
\begin{equation}\langle  \A \phi,\phi\rangle_{\Hi} + \langle  \phi,\A \phi\rangle_{\Hi} = - \Vert \phi^\prime \Vert^2_{\Hi}.
\end{equation}
Using the Poincar\'e inequality, inequality \eqref{eq:lyap-ineq} follows with $\mu=\pi$.
Note also that the couple $(S,\Gamma)$ satisfies Assumption \ref{ass:marginallystable} with $Q=\I_r$.
Moreover, since $S$ is skew adjoint, its eigenvalues are on the imaginary axis, then  Assumption~\ref{ass:disjoint} is satisfied since  $\A$ contains eigenvalues at the left hand side of the imaginary axis.

\subsection{Construction of $\M$ and the feedback law}
To apply the control law \eqref{eq:feedback-law}-\eqref{eq:H}, we look for an operator ${\cal M}:L^2(0,L)\to \RR^{r}$ solution to the Sylvester equation \eqref{eq:Sylvester} which in our context becomes
\begin{equation}
    \label{eq_M_sylvester_chaleur}
\M \phi^{\prime\prime}= S \M \phi + \Gamma \phi(\ell)  \,, \quad \forall \phi\in D(\A).
\end{equation}
Using Lemma \ref{lemma:M}, one can write $\mathcal{M}$ as follows:
\begin{equation}
    \label{eq_M_integral}
\M \phi = \int_0^1 M(x)\phi(x)dx,
\end{equation}
where $M:[0,1]\mapsto\RR^r$ is a continuous function defined as
\begin{equation}
M(x)=\left\{
\begin{array}{l}
E_1\exp(Fx)N_0\,, \hfill x\in (0,\ell)
\\[.5em] 
    E_1\exp(Fx)N_0 + E_1\exp(F(x-\ell))G\,, \\ 
\hfill x\in (\ell,1)
\end{array}
\right.
\end{equation}
where $(F,G,N_0,E_1)$ are matrices respectively in $\RR^{2r\times 2r}$, $\RR^{2r}$, $\RR^{2r}$ and $\RR^{2r}$ defined as
\begin{equation}
\begin{aligned}
    F &= \begin{bmatrix}
 0 & I_r \\ S & 0
 \end{bmatrix}
  , 
 \quad
 G = 
 \begin{bmatrix}
 0 \\ \Gamma
 \end{bmatrix},\quad 
 E_1 &= \begin{bmatrix}
I_r&0
\end{bmatrix},
 \end{aligned}
\end{equation}
and
\begin{equation}
N_0 = 
-\left(
\begin{bmatrix}
E_1\\
E_1\exp(F)
\end{bmatrix}\right) ^{-1}
\begin{bmatrix}
0\\
-E_1\exp(F(1-\ell))G
\end{bmatrix}.
\end{equation}
In this case, it can be shown that the operator given in \eqref{eq_M_integral} satisfies \eqref{eq_M_sylvester_chaleur}.
First of all, note that $M(0)=0$ and that with
\begin{equation}
\begin{bmatrix}
E_1\\
E_1\exp(F)
\end{bmatrix}N_0 + 
\begin{bmatrix}
0\\
-E_1\exp(F(1-\ell))G
\end{bmatrix}=0\ ,
\end{equation}
it yields $M(1)=0$.

Note also that for all $\phi$ in $D(\A)$, using $M(0)=0$, the fact that $M$ is $C^2$ in $(0,\ell)$ and by integration by part twice,
\begin{multline}
\int_0^\ell M(x)\phi^{\prime\prime}(x)dx = S\int_0^\ell M(x)\phi^{\prime\prime}(x)dx \\
+ M(\ell)\phi^\prime(\ell) - E_2\exp(F\ell)N_0\phi(\ell)\ .
\end{multline}
where $E_2=E_1F=\begin{bmatrix}0&\I_r\end{bmatrix}$.

On the other hand, using $M(1)=0$, the fact that $M$ is $C^2$ in $(\ell,1)$ and by integration by part twice,
\begin{multline}
\int_\ell^1 M(x)\phi^{\prime\prime}(x)dx = \\ S\int_\ell^1 M(x)\phi(x)dx+E_2[\exp(F\ell)N_0+G]\phi(\ell)
\\ 
- M(\ell)\phi^\prime(\ell)\ .
\end{multline}
%
Finally, since $E_2G=\Gamma$, it yields
$$
\int_0^1 M(x)\phi^{\prime\prime}(x)dx = S\int_0^1 M(x)\phi(x)dx + \Gamma \phi(\ell)\ ,
$$
which implies that \eqref{eq_M_integral} is satisfied.

Hence, Theorem \ref{thm:wp} ensures that the system given in (\ref{eq_chaleur}) with the control law
\begin{equation}\label{eq_ContrLawHeat}
u(t) = -k H z(t), \quad  H = \left(\int_0^1M(x)b(x)dx\right)^\top,    
\end{equation}
 is well-posed in $\W=(\RR^r\times\Hi)$ for all $k$ in $(0,k^*)$ where $k^*$ can simply be computed  as in \eqref{eq:kstar}.

\subsection{Stabilization of the origin}
 Consider the case in which $r=2$ and $S,\Gamma$ are given by
 \begin{equation}
    S:=\begin{bmatrix}
    0 & 1\\
    -1 & 0
    \end{bmatrix},\qquad \Gamma :=\begin{bmatrix}
    1 & 0
    \end{bmatrix}\ .
 \end{equation}
Due to the particular structure of the matrix $S$, the pair $(S,H)$ is observable if and only if $H\neq 0$.
 Hence, if $H$ defined in \eqref{eq_ContrLawHeat} is non zero,  Theorem \ref{thm:as} ensures exponential stability of the origin of the system (\ref{eq_chaleur}) with the control law \eqref{eq_ContrLawHeat}.

\section{Conclusions}

\label{sec_conclusion}

We have provided in this article a design method for the exponential stabilization of an infinite-dimensional system (supposed to be exponentially stable) coupled with an ODE (supposed to be marginally stable). This design is based on the forwarding method. The exponential stability has been proved thanks to a strictification technique. As further research lines to be followed, let us mention the case of unbounded control operators, which will imply a more sophisticated analysis. We believe also that this approach could be applied to nonlinear systems. 

\bibliographystyle{plain}
\bibliography{bibsm}

\begin{thebibliography}{10}

\bibitem{astolfi2021repetitive}
D.~Astolfi, S.~Marx, and N.~van~de Wouw.
\newblock Repetitive control design based on forwarding for nonlinear
  minimum-phase systems.
\newblock {\em Automatica}, 129:109671, 2021.

\bibitem{astolfi2021nonlinear}
D.~Astolfi, L.~Praly, and L.~Marconi.
\newblock {Nonlinear Robust Periodic Output Regulation of Minimum Phase
  Systems}.
\newblock {\em Math. Control Signals Syst.
  https://doi.org/10.1007/s00498-021-00307-w}, 2021.

\bibitem{balogoun2021iss}
I.~Balogoun, S.~Marx, and D.~Astolfi.
\newblock {ISS Lyapunov strictification via observer design and integral action
  control for a Korteweg-de-Vries equation}.
\newblock {\em arXiv preprint arXiv:2107.09541}, 2021.

\bibitem{bekiaris2014compensation}
N.~Bekiaris-Liberis and M.~Krstic.
\newblock Compensation of wave actuator dynamics for nonlinear systems.
\newblock {\em IEEE Transactions on Automatic Control}, 59(6):1555--1570, 2014.

\bibitem{bribiesca}
D.~Bou~Saba, F.~Bribiesca-Argomedo, M.~Michael Di~Loreto, and D.~Eberard.
\newblock Strictly proper control design for the stabilization of 2x2 linear
  hyperbolic {ODE}-{PDE}-{ODE} systems.
\newblock In {\em Proceedings of the 58th Conference on Decision and Control},
  pages 4996--5011, 2019.

\bibitem{coron2019pi}
J.-M. Coron and A.~Hayat.
\newblock {PI controllers for 1-D nonlinear transport equation}.
\newblock {\em IEEE Transactions on Automatic Control}, 64(11):4570--4582,
  2019.

\bibitem{davison1975new}
E~Davison and S~Wang.
\newblock New results on the controllability and observability of general
  composite systems.
\newblock {\em IEEE Transactions on Automatic Control}, 20(1):123--128, 1975.

\bibitem{deutscher2015backstepping}
J.~Deutscher.
\newblock A backstepping approach to the output regulation of boundary
  controlled parabolic {PDEs}.
\newblock {\em Automatica}, 57:56--64, 2015.

\bibitem{deutscher2019robust}
J.~Deutscher and S.~Kerschbaum.
\newblock Robust output regulation by state feedback control for coupled linear
  parabolic {PIDEs}.
\newblock {\em IEEE Transactions on Automatic Control}, 65(5):2207--2214, 2019.

\bibitem{jankovic2009forwarding}
M.~Jankovic.
\newblock Forwarding, backstepping, and finite spectrum assignment for time
  delay systems.
\newblock {\em Automatica}, 45(1):2--9, 2009.

\bibitem{krstic2009compensating}
M.~Krstic.
\newblock Compensating actuator and sensor dynamics governed by diffusion
  {PDE}s.
\newblock {\em Systems \& Control Letters}, 58(5):372--377, 2009.

\bibitem{krstic2008backstepping}
M.~Krstic and A.~Smyshlyaev.
\newblock Backstepping boundary control for first-order hyperbolic {PDE}s and
  application to systems with actuator and sensor delays.
\newblock {\em Systems \& Control Letters}, 57(9):750--758, 2008.

\bibitem{marx2017cone}
S.~Marx, V.~Andrieu, and C.~Prieur.
\newblock {Cone-bounded feedback laws for $m$-dissipative operators on Hilbert
  spaces}.
\newblock {\em Mathematics of Control, Signals, and Systems}, 29(4):1--32,
  2017.

\bibitem{marx2021forwarding}
S.~Marx, L.~Brivadis, and D.~Astolfi.
\newblock {Forwarding techniques for the global stabilization of dissipative
  infinite-dimensional systems coupled with an ODE}.
\newblock {\em Mathematics of Control, Signals, and Systems}, pages 1--20,
  2021.

\bibitem{mazenc1996adding}
F.~Mazenc and L.~Praly.
\newblock Adding integrations, saturated controls, and stabilization for
  feedforward systems.
\newblock {\em IEEE Transactions on Automatic Control}, 41(11):1559--1578,
  1996.

\bibitem{mironchenko2019non}
A.~Mironchenko and F.~Wirth.
\newblock {Non-coercive Lyapunov functions for infinite-dimensional systems}.
\newblock {\em Journal of Differential Equations}, 266(11):7038--7072, 2019.

\bibitem{paunonen2015controller}
L.~Paunonen.
\newblock Controller design for robust output regulation of regular linear
  systems.
\newblock {\em IEEE Transactions on Automatic Control}, 61(10):2974--2986,
  2015.

\bibitem{paunonen2014internal}
L.~Paunonen and S.~Pohjolainen.
\newblock The internal model principle for systems with unbounded control and
  observation.
\newblock {\em SIAM Journal on Control and Optimization}, 52(6):3967--4000,
  2014.

\bibitem{pazy1983semigroups}
A.~Pazy.
\newblock {\em Semigroups of linear operators and applications to partial
  differential equations}.
\newblock Springer, 1983.

\bibitem{Phong1991}
V.~Q. Ph\'ong.
\newblock {The operator equation $AX - XB = C$ with unbounded operators $A$ and
  $B$ and related abstract Cauchy problems.}
\newblock {\em Mathematische Zeitschrift}, 208(4):567--588, 1991.

\bibitem{praly2019observers}
L.~Praly.
\newblock {Observers to the aid of ``strictification” of Lyapunov functions}.
\newblock {\em Systems $\&$ Control Letters}, 134:104510, 2019.

\bibitem{walter1987real}
W.~Rudin.
\newblock {\em Real and complex analysis}.
\newblock McGraw-Hill Book Company, 1987.

\bibitem{tang2011state}
S.~Tang and C.~Xie.
\newblock State and output feedback boundary control for a coupled {PDE}--{ODE}
  system.
\newblock {\em Systems \& Control Letters}, 60(8):540--545, 2011.

\bibitem{terrand2019adding}
A.~Terrand-Jeanne, V.~Andrieu, V.~Dos Santos~Martins, and C.-Z. Xu.
\newblock Adding integral action for open-loop exponentially stable semigroups
  and application to boundary control of {PDE} systems.
\newblock {\em IEEE Transactions on Automatic Control}, 2019.

\bibitem{tucsnak2009observation}
M.~Tucsnak and G.~Weiss.
\newblock {\em Observation and control for operator semigroups}.
\newblock Springer, 2009.

\bibitem{xu1995robust}
C.-Z. Xu and H.~Jerbi.
\newblock A robust {PI}-controller for infinite-dimensional systems.
\newblock {\em International Journal of Control}, 61(1):33--45, 1995.

\bibitem{zhang2019pi}
L.~Zhang, C.~Prieur, and J.~Qiao.
\newblock {PI boundary control of linear hyperbolic balance laws with
  stabilization of ARZ traffic flow models}.
\newblock {\em Systems $\&$ Control Letters}, 123:85--91, 2019.

\end{thebibliography}

\end{document}